\documentclass[twoside,a4paper,12pt,leqno]{amsart}

\usepackage{verbatim,amsxtra,amssymb,amsmath,color,psfrag,graphicx,mathrsfs}

\usepackage[bookmarksnumbered,colorlinks,plainpages,backref,pagebackref,pdfstartview=FitBH]{hyperref}

\definecolor{mygrey}{gray}{0.70}
\definecolor{mygreen}{rgb}{0,.75,0}
\definecolor{myred}{rgb}{1,0,0}
\definecolor{orange}{rgb}{1,.5,0}

\numberwithin{equation}{section}

\newtheorem{thm}[equation]{Theorem}

\newtheorem{cor}[equation]{Corollary}

\newtheorem{prop}[equation]{Proposition}

\theoremstyle{definition}
\newtheorem{defn}[equation]{Definition}

\theoremstyle{remark}
\newtheorem{rem}[equation]{Remark}

\newcommand{\thmref}[1]{Theorem~\ref{#1}}

\newcommand{\corref}[1]{Corollary~\ref{#1}}
\newcommand{\defref}[1]{Definition~\ref{#1}}
\newcommand{\remref}[1]{Remark~\ref{#1}}

\newcommand{\figref}[1]{Figure~\ref{#1}}
\newcommand{\secref}[1]{Section~{\bf\ref{#1}}}

\renewcommand\b{{\beta}}

\newcommand\D{{\Delta}}

\newcommand\DD{{\boldsymbol\Delta}}
\renewcommand\deg{\operatorname{\mathsf{deg}}}
\newcommand\diag{\operatorname{\mathsf{diag}}}
\newcommand\du{\mathop{\uparrow\!\downarrow}}

\newcommand\E{\mathbf E}
\newcommand\e{{\boldsymbol \varepsilon}}

\newcommand\f{{\varphi}}
\newcommand\ft[1]{{\parbox{130truemm}{#1}}}

\newcommand\G{{\boldsymbol G}}
\newcommand\g{\gamma}
\newcommand\Ga{\Gamma}
\newcommand\GG{{\mathcal G}}

\newcommand\Iso{\operatorname{\mathsf{Iso}}}

\newcommand\K{\mathcal K}
\renewcommand\ker{\operatorname{\mathsf{ker}}}
\newcommand\KK{{\boldsymbol K}}

\renewcommand\l{\lambda}

\newcommand\N{{\mathbb N}}
\renewcommand\ni{{\not\circlearrowright}}
\newcommand\Nu{{\boldsymbol \#}}

\newcommand\oo{{\boldsymbol o}}
\newcommand\ov{\overline}

\renewcommand\P{\mathbf P}
\newcommand\pt{\partial}

\newcommand\Q{\mathbf Q}

\newcommand\R{{\mathcal R}}
\newcommand\RR{{\boldsymbol R}}

\renewcommand\S{{\mathscr S}}
\newcommand\s{\operatorname{\boldsymbol s}}
\newcommand\si{\sigma}

\renewcommand\t{\operatorname{\boldsymbol t}}
\newcommand\T{{\mathcal T}}
\newcommand\Tb{{\mathbb T}}
\renewcommand\th{{\boldsymbol \theta}}
\newcommand\toto{\mathop{\;\longrightarrow\;}}

\newcommand\TT{\boldsymbol T}

\newcommand\wt{\widetilde}

\newcommand\Z{\mathbb Z}

\begin{document}

\title[Random horospheric products]{Stochastic homogenization of
horospheric tree products}

\author[V. A. Kaimanovich]{Vadim A. Kaimanovich}

\address{Mathematics, Jacobs University Bremen, Campus Ring 1, D-28759 Bremen, Germany}

\email{v.kaimanovich@jacobs-university.de, vadim.kaimanovich@gmail.com}

\author{Florian Sobieczky}

\address{Mathematics C, Graz University of Technology, Steyrergasse 30, A-8010 Graz,
Austria \textup{and} Mathematisches Institut, Universit\"at Jena, Ernst Abbe
Platz 2, D-07743 Jena, Germany}

\email{sobieczky@tugraz.at, Florian.Sobieczky@uni-jena.de}

%\date{\today}

\begin{abstract}
We construct measures invariant with respect to equivalence relations which
are graphed by horospheric products of trees. The construction is based on
using conformal systems of boundary measures on treed equivalence relations.
The existence of such an invariant measure allows us to establish amenability
of horospheric products of random trees.
\end{abstract}

\maketitle

\thispagestyle{empty}

\section*{Introduction}

The study of \emph{graphed measured equivalence relations} has two origins.
The first one is \emph{ergodic}, namely, the orbit equivalence theory for
measure class preserving actions of countable groups. The second one is
\emph{geometric}, because such equivalence relations naturally arise (as
traces of the leaf partition) on transversals of foliated or laminated spaces
endowed with holonomy (quasi-)invariant measures.

The departure point of Feldman and Moore in their famous paper
\cite{Feldman-Moore77} published in 1977 (and announced two years earlier
\cite{Feldman-Moore75}) was entirely ergodic: it was the idea that numerous
properties of measure class preserving group actions can actually be expressed
just in terms of the associated orbit equivalence relation. In their work
Feldman and Moore did not consider any additional leafwise graph structures on
equivalence relations. However, at about the same time Plante \cite{Plante75}
essentially introduced graphed equivalence relations (in terms of finitely
generated holonomy pseudogroups) in the topological context of foliations.

It was only in 1990 that Adams \cite{Adams90} defined the notion of a graphed
equivalence relation in the purely measure-theoretical setup and proved
non-amenability of non-elementary treed equivalence relations with a finite
invariant measure. Later this notion was used by the first author
\cite{Kaimanovich97} in order to clarify the relationship between the
amenability of an equivalence relation and the amenability of its leafwise
graphs and to give a new geometrical proof of the Connes--Feldman--Weiss
theorem on the equivalence of hyperfiniteness and amenability. A new insight
was brought in by Gaboriau \cite{Gaboriau00} by introducing the \emph{cost} of
an equivalence relation with an invariant probability measure as the lowest
possible average value of the degree of vertices in leafwise connected graph
structures. This invariant turned out to be very useful and has found numerous
applications (e.g., see the recent survey \cite{Furman09}).

\medskip

From the probabilistic point of view a discrete equivalence relation with a
quasi-invariant measure naturally arises from any stationary Markov chain with
discrete transition probabilities (see \cite{Kaimanovich98} for an explicit
formula for the Radon--Nikodym cocycle). Finitely supported transition
probabilities then produce a locally finite graph structure on this
equivalence relation. In particular, there is a canonical one-to-one
correspondence between \emph{detailed balance} stationary measures of the
leafwise simple random walk on the state space of a graphed equivalence
relation (i.e., the ones with respect to which this random walk is reversible)
and invariant measures of the equivalence relation. Namely, the density of a
stationary measure with respect to the corresponding invariant measure is just
the vertex degree function $\deg$.

Another link with the probability theory is provided by the fact that a
graphed equivalence relation on a probability space naturally gives rise to a
map from this space to the space of rooted graphs $\GG$ (i.e., a \emph{random
rooted graph}). It assigns to any point from the state space its leafwise
graph with this very point as the distinguished vertex. \emph{Stochastic
homogenization} \cite{Kaimanovich03a} of a certain family of infinite graphs
consists in finding a probability measure invariant with respect to an
equivalence relation whose classes are endowed with graph structures from this
family. The role of such a measure is then similar to the role of an invariant
measure for a usual dynamical system.

The space of rooted graphs itself has an intrinsic graph structure: two rooted
graphs are neighbours if they are isomorphic as unrooted graphs and their
roots are neighbours in this common graph $\Ga$. Moreover, if $\Ga$ is
\emph{rigid}, i.e., its isometry group is trivial, then the graph on its
equivalence class (which is obtained by varying the root position) is
precisely $\Ga$ itself. Denote by $\GG_\ni\subset\GG$ the space of rooted
rigid graphs. Then for rigid graphs the problem of stochastic homogenization
reduces to finding an invariant probability measure on the corresponding
subspace of $\GG_\ni$.

\medskip

It is easy to construct invariant measures on $\GG_\ni$ by random
perturbations of Cayley graphs of finitely generated infinite groups. However,
there are random graphs whose origin has nothing to do with groups. The first
example of stochastic homogenization in such situations is the invariant
measure on the space of rigid rooted trees obtained from \emph{augmented
branching processes}. When studying random walks on Galton--Watson trees
Lyons, Pemantle and Peres noticed that it is natural to modify the branching
process by letting the progenitor to have one more additional offspring (so
that all vertices statistically have the same number of neighbours). Then the
arising probability measure on Galton--Watson trees augmented in this way is
stationary with respect to the leafwise simple random walk
\cite{Lyons-Pemantle-Peres95}. Thus, by \cite{Kaimanovich98}, dividing this
measure by the degree function $\deg$ produces an invariant measure (in
\thmref{th:inv} we also give a simple direct proof of this fact).

\medskip

The main purpose of the present paper is to obtain a stochastic homogenization
for yet another family of graphs: \emph{horospheric products of trees.}

\medskip

These graphs were first introduced by Diestel and Leader
\cite{Diestel-Leader01} in an attempt to answer a question of Woess
\cite{Woess91} on existence of vertex-transitive graph not quasi-isometric to
Cayley graphs. Although the fact that the Diestel--Leader graphs indeed
provide such an example was only recently proved by Eskin, Fisher and Whyte
\cite{Eskin-Fisher-Whyte07}, in the meantime the construction of Diestel and
Leader attracted a lot of attention because of its numerous interesting
features (see \cite{Woess05,Bartholdi-Neuhauser-Woess08} and the references
therein).

The starting point of this construction is the fact that, given an infinite
tree $T$, any boundary point $\g\in\pt T$ determines the associated
$\Z$-valued additive \emph{Busemann cocycle} $\b_\g$ on $T$: for any two
vertices $x,y\in T$ the value $\b_\g(x,y)$ is, informally speaking, the
``difference between the distances'' from the points $x$ and $y$ to the point
at infinity $\g$ (this cocycle is actually well-defined for any
$\mathsf{CAT}(0)$ space). The level sets of the Busemann cocycle $\b_\g$
consist of the points in $T$ which are \emph{equidistanced} from $\g$ and are
called \emph{horospheres} (or \emph{horocycles} in the case of the classical
hyperbolic plane, whence the frequently used alternative term ``horocyclic
products'').

The horospheric product of two pointed at infinity rooted trees $(T,o,\g)$ and
$(T',o',\g')$ is then defined in the following way. Take the graph-theoretical
product of the trees $T,T'$, and consider its subgraph $\Ga$ which has the
same vertex set $T\times T'$, but contains only those edges of the product
graph which are in the kernel of the cocycle $\b_\g+\b_{\g'}$. The horospheric
product of the trees $(T,o,\g)$ and $(T',o',\g')$ is then the connected
component of the graph $\Ga$ which contains the product origin $(o,o')$.

Geometrically one can think about the horospheric products in the following
way \cite{Kaimanovich-Woess02}. Draw the tree $T'$ upside down next to $T$ so
that the respective horospheres are at the same levels. Connect the two roots
$o,o'$ with an elastic spring. It can move along each of the two trees, may
expand infinitely, but must always remain in horizontal position. The vertex
set of the horospheric product consists then of all admissible positions of
the spring. From a position $(x,x')$ with $\b_\g(o,x)+\b'_{\g'}(o',x')=0$ the
spring may move downwards to one of the ``sons'' of $x$ and at the same time
to the ``father'' of $x'$, or upwards in an analogous way. Such a move
corresponds to going to a neighbour $(y,y')$ of $(x,x')$.

\medskip

It is natural to look for a stochastic homogenization of horospheric products
based on stochastic homogenizations of trees, i.e., on treed equivalence
relations with an invariant probability measure. However, such equivalence
relations are non-amenable (unless elementary) \cite{Adams90}, and therefore
there is no measurable way of assigning a single boundary point $\g\in\pt T_x$
to any point $x$ from the base space of the equivalence relation
\cite{Kaimanovich04} (here $T_x$ is the leafwise tree on the equivalence class
of $x$). Thus, a stochastic homogenization of horospheric products should be
preceded by a choice of an appropriate measurable \emph{system of boundary
measures} $\{\nu_x\}$ on $\pt T_x$. By analogy with Fuchsian and Kleinian
groups (e.g., see \cite{Patterson76,Sullivan79}) we say that a system
$\{\nu_x\}$ is \emph{conformal} if it is quasi-invariant and its
Radon--Nikodym derivatives satisfy the relation $d\nu_y/d\nu_x(\g) =
\exp(-\l\b_\g(x,y))$ for a certain $\l>0$ (the \emph{dimension} of the
system).

Given a treed equivalence relation $R$ on a space $X$, we define its
\emph{boundary bundle} as the set $\wt X=\{(x,\g):x\in X,\g\in\pt T_x\}$ (this
is an immediate analogue of the unit tangent bundle on negatively curved
manifolds). If two points $x,y\in X$ are equivalent, then $T_x$ and $T_y$
coincide as unrooted trees, so that there is a natural identification of the
boundaries $\pt T_x$ and $\pt T_y$. Therefore, $\wt X$ also has a structure of
a treed equivalence relation
$$
\wt R=\{(x,y,\g): (x,y)\in R,\g\in\pt T_x\cong\pt T_y\} \;,
$$
and $\wt R$ carries the $\Z$-valued additive cocycle $\wt\b=\b_\g(x,y)$. If
$\mu$ is an $R$-invariant measure on $X$, and $\{\nu_x\}$ is a conformal
system of boundary measures of dimension $\l$, then the result of the
integration of the system $\{\nu_x\}$ against the measure $\mu$ is the measure
$\wt\mu$ on $\wt X$ such that its Radon--Nikodym cocycle with respect to the
equivalence relation $\wt R$ is precisely $\exp(-\l\wt\b)$. Thus, the measure
$\wt\mu$ is, in our context, a solution of a classical problem of ergodic
theory: to find a measure with the prescribed Radon--Nikodym derivatives.

Now, let $\{\nu'_{x'}\}$ be a conformal system of boundary measures of the
same dimension $\l$ on another treed equivalence relation with invariant
measure $(X',\mu',R')$, and let $\wt\mu'$ be the associated measure on the
boundary bundle $\wt X'$. We shall say that the kernel $\RR$ of the cocycle
$c=\wt\b+\wt\b'$ is the \emph{horospheric product} of the treed equivalence
relations $R$ and $R'$. The equivalence relation $\RR$ is endowed with a
natural graph structure such that its equivalence classes are precisely the
horospheric products of trees from equivalence relations $R$ and $R'$.
Moreover, the product measure $\wt\mu\times\wt\mu'$ is $\RR$-invariant, thus
providing the sought for stochastic homogenization of horospheric products
(\thmref{th:hp}). It follows from the fact that the logarithms of the
Radon--Nikodym cocycles of the measures $\wt\mu,\wt\mu'$ are proportional to
the respective Busemann cocycles $\wt\b,\wt\b'$ with the same proportionality
coefficient $-\l$, so that the product measure is invariant with respect to
the kernel of $\wt\b+\wt\b'$. Note that this construction is very similar to
the construction of an invariant measure of the geodesic flow on a negatively
curved manifold from a conformal measure \cite{Kaimanovich90}.

\medskip

As an application we show in \thmref{th:am} that the horospheric product of
almost any pair of pointed at infinity rooted trees arising in the above
situation is amenable, (i.e., does not satisfy the strong isoperimetric
inequality: there are subsets whose boundary is arbitrarily small compared
with the subset itself). The proof is based on the fact that the equivalence
relations $(\wt X,\wt\mu,\wt R)$ and $(\wt X',\wt\mu',\wt R')$ are both
amenable (as they are graphed by trees pointed at infinity). Therefore their
product and its subrelation $\RR$ are also amenable. On the other hand, since
$\RR$ has a finite invariant measure, amenability of $\RR$ implies amenability
of its leafwise graphs.

Another application is the existence of the associated finite stationary
measure of the leafwise simple random walk and the ensuing possibility for a
study of the asymptotical properties of simple random walks on individual
horospheric products (the linear rate of escape, the harmonic measure, the
Poisson boundary, the asymptotic entropy etc.). We shall return to this
subject in another publication.

\medskip

As an example in \secref{sec:3} we consider the horospheric products of
augmented Galton--Watson trees. It is easy to see that the \emph{branching
measures} on their boundaries (i.e., the limits of the appropriately
normalized uniform measures on the spheres around the root) for a conformal
system of dimension $\l=\log m$, where $m$ is the mean of the offspring
distribution. Thus, horospheric products of augmented Galton--Watson trees
corresponding to any two branching processes with the same mean are
stochastically homogeneous.

\medskip

There are numerous natural questions which arise in connection with our study.
We hope to address them in the future, and here we shall just briefly mention
some of them.

\medskip

(1) In the present paper we do not consider at all the question about the
ergodicity of arising measures. In fact one can show that in our setup the
horospheric product of ergodic boundary bundles is also ergodic (the proof is
based on an analogue of the famous \emph{Hopf argument} used for proving
ergodicity of the geodesic flow on negatively curved manifolds
\cite{Kaimanovich90}).

(2) Currently the augmented Galton--Watson measures (and similar measures
arising from more general branching processes) are the only examples of
``nice'' invariant measures on the space of rooted trees. Although it was
recently proved that any invariant measure can be obtained as an appropriate
weak limit \cite{Elek08}, it would still be interesting to have other explicit
examples.

(3) Which treed equivalence relations with an invariant measure admit a
conformal system of boundary measures? When is such a system unique? In fact,
conformal systems are closely related with the Hausdorff boundary measures (in
perfect analogy with the Fuchsian and Kleinian case \cite{Sullivan79}). One
can show that under natural assumptions there is at most one conformal system
of boundary measures which coincides with the system of the Hausdorff
measures.

(4) Our point of view on boundary measures on random trees consists in
considering \emph{systems} of boundary measures corresponding to varying roots
rather than a \emph{single} measure (once again, in perfect analogy with the
theory of boundary measures on negatively curved manifolds). In addition to
the Busemann cocycle one can consider other natural cocycles (or potentials)
on the boundary bundle and ask for existence of boundary systems with
prescribed Radon--Nikodym derivatives, which leads to the notion of a
\emph{Gibbs system of boundary measures}. This notion, in particular, provides
a unified approach to a number of results on multifractal properties of
various boundary measures on Galton--Watson trees
\cite{Morters-Shieh04,Kinnison08}. It is also interesting to look at the
ergodic properties of the arising invariant measures of the leafwise geodesic
flow.

(5) It is still unknown whether a.e. leafwise graph in a graphed equivalence
relation with a finite invariant measure has a precise exponential rate of
growth \cite{Hurder-Katok87}. As far as we know, this issue is completely open
even for treed equivalence relations. A refinement of this question is the
following problem: for which treed equivalence relations do the normalized
uniform measures on spheres converge (like for the Galton--Watson and other
trees arising from branching processes)?

\medskip

The first author would like to thank the organizers of the First Seasonal
Institute of the Mathematical Society of Japan ``Probabilistic Approach to
Geometry'' for their warm hospitality and excellent working conditions. The
work of the second author was supported by the Austrian science fund (FWF)
under the project number P18703.

\section{Graphed equivalence relations} \label{sec:1}

\subsection{Equivalence relations} \label{sec:equiv}

We shall first remind the basics from the theory of \emph{discrete measured
equivalence relations} created by Feldman and Moore \cite{Feldman-Moore77}.
Their starting point was the observation that many properties of a measure
class preserving action of a countable group can actually be expressed just in
terms of the corresponding orbit equivalence relation. We shall partially use
the \emph{groupoid} approach, see \cite{Renault80,Anantharaman-Renault00}.

For an arbitrary equivalence relation $R\subset X\times X$ on a \emph{state
space} $X$ the \emph{composition}
%$$
\begin{equation} \label{eq:comp}
(x,y)(y,z)=(x,z) \qquad \text{for}\quad (x,y),(y,z) \in R
\end{equation}
%$$
determines a \emph{groupoid structure} $\G=\G(R)$ with
\begin{itemize}
\item
the \emph{set of objects}~$X$,
\item
the \emph{set of morphisms} $R$,
\item
the \emph{source map} $\s:(x,y)\mapsto x$,
\item
the \emph{target map} $\t:(x,y)\mapsto y$,
\item
the \emph{identity embedding} $\e:x\mapsto(x,x)$,
\item
the \emph{involution} $\th:(x,y)\mapsto(x,y)^{-1}=(y,x)$.
\end{itemize}
Denote by
$$
[x]=[x]_R=R(x)
$$
the \emph{$R$-equivalence class} of a point $x\in X$. In other terminologies
(which come from two important sources of equivalence relations: foliations
and group actions) one also calls equivalences classes \emph{leafs} or
\emph{orbits}.

An equivalence relation $R$ on $X$ is called \emph{discrete measured} if
\begin{itemize}
   \item[(i)]
It is \emph{countable}, i.e., the classes $[x]$ are at most countable;
    \item[(ii)]
Its state space $X$ is endowed with a structure of a \emph{standard Borel
space}, and it carries a $\si$-finite Borel measure $\mu$, so that $(X,\mu)$
is a \emph{Lebesgue measure space} (i.e., its non-atomic part is isomorphic to
an interval with the Lebesgue measure on it);
   \item[(iii)]
It is measurable as a subset of $X\times X$ (endowed with the product Borel
structure);
    \item[(iv)]
It preserves the class of the measure $\mu$ ($\equiv$ the measure $\mu$ is
\emph{quasi-invariant} with respect to $R$), which means that for any subset
$A\subset X$ with $\mu(A)=0$ its \emph{saturation}
$$
[A]=\bigcup_{x\in A}[x]
$$
also has measure 0.
\end{itemize}

Below all the equivalence relations are assumed to be discrete measured with
\emph{infinite equivalence classes}. All the properties related to measure
spaces will be understood \emph{mod~0}, i.e., up to measure 0 subsets.

Any discrete measured equivalence relation can be presented as the orbit
equivalence relation of a measure class preserving action of a certain
countable group (although there are equivalence relations for which such an
action can not be free \cite{Furman99}). However, there are equivalence
relations whose origin \emph{a priory} has nothing to do with group actions
(for instance, \emph{treed equivalence relations} which we shall study below).

\subsection{The Radon--Nikodym cocycle} \label{sec:RN}

The fibers of the source map $\s$ satisfy the \emph{transitivity relation}:
$$
(x,y)\s^{-1}(y) =\s^{-1}(x) \qquad\forall\,(x,y)\in R
$$
(the multiplication in the left-hand side is the groupoid composition
\eqref{eq:comp}). Denote by $\Nu_x$ the \emph{counting measure} on the fiber
$\s^{-1}(x)$  of the source map (this fiber is in obvious one-to-one
correspondence $(x,y)\mapsto y$ with the class $[x]$). The system of measures
$\{\Nu_x\}_{x\in X}$ is then \emph{left invariant} in the sense that
$$
(x,y)\Nu_y=\Nu_x \qquad\forall\,(x,y)\in R \;,
$$
so that it is a \emph{source} (or \emph{left}) \emph{Haar system} for the
groupoid $\G$. The result of the integration of the fiber measures $\Nu_x$
against the measure $\mu$ on the state space $X$ is the $\si$-finite
measure~$\mu_\#$ defined as
$$
d\mu_\#(x,y) = d\mu(x) d\Nu_x(y) = d\mu (x) \;,
$$
which is called the \emph{left counting measure} on $R$.

In the same way, denote by $\Nu^x$ the counting measure on the fiber
$\t^{-1}(x)$ of the target map. The system $\{\Nu^x\}$ is \emph{right
invariant} in the sense that
$$
\Nu^x(x,y)=\Nu^y \qquad\forall\,(x,y)\in R \;,
$$
so that it is a \emph{target} (or \emph{right}) \emph{Haar system} for the
groupoid $\G$. The result of the integration of the fiber measures $\Nu^x$
against the measure $\mu$ on the state space $X$
$$
d\mu^\#(x,y) = d\mu(y) d\Nu^y(x) = d\mu (y)
$$
is called the \emph{right counting measure} on $R$. Alternatively, the right
counting measure~$\mu^\#$ can be obtained from the left counting measure
$\mu_\#$ (and \emph{vice versa}) by applying the involution $\th$, so that
$\mu^\#=\th\mu_\#$ and $\mu_\#=\th\mu^\#$.

It turns out that the measures $\mu_\#$ and $\mu^\#$ are equivalent if and
only if the original measure $\mu$ is quasi-invariant with respect to $R$. In
this case the Radon--Nikodym derivative
$$
\D(x,y) = \frac{d\mu^\#}{d\mu_\#} (x,y)
$$
is called the \emph{Radon--Nikodym cocycle} of the measure $\mu$ with respect
to $R$ (it is a \emph{multiplicative cocycle} in the sense that
$$
\D(x,y) \D(y,z) = \D(x,z)
$$
for any triple of equivalent points $x,y,z\in X$). If $\D\equiv 1$, then the
measure $\mu$ is \emph{$R$-invariant}, or, respectively, the equivalence
relation $R$ \emph{preserves} the measure $\mu$.

Equivalently, the measure $\mu$ is quasi-invariant with respect to $R$ if and
only if for any \emph{partial transformation} $\f$ of $R$ (i.e., a measurable
bijection between two measurable subsets $A,B\subset X$ whose graph is
contained in $R$) the $\f$-image $\f(\mu|_A)$ of the restriction of $\mu$ to
$A$ is absolutely continuous with respect to the restriction $\mu|_B$ of $\mu$
to $B$, and
$$
\D(x,y) = \frac{d\f^{-1}\mu}{d\mu} (x) = \frac{d\mu}{d\f\mu}(y) \;.
$$
Thus, the Radon--Nikodym cocycle can also be considered as the ``ratio of
differentials''
$$
\D(x,y) = \frac{d\mu(y)}{d\mu(x)} \;, \qquad (x,y)\in R \;.
$$
This formalism is quite convenient and can always be made rigorous by passing
to the appropriate partial transformations.

If $R=R_G$ is the orbit equivalence relation determined by a measure class
preserving action of a countable group $G$ on a measure space $(X,\mu)$, then
$$
\D(x,gx)=\frac{dg^{-1}\mu}{d\mu}(x) \;.
$$

\subsection{Graph structures} \label{sec:str}

Recall that a \emph{graph} $\Ga$ is determined by its \emph{set of vertices}
(usually it is denoted in the same way as the graph itself) and its \emph{set
of edges}. We shall always deal with \emph{non-oriented graphs without loops
and multiple edges}, so that the set of edges can be identified with a
symmetric subset of $\Ga\times\Ga\setminus\diag$.

Analogously, a (non-oriented) \emph{graph structure} on a discrete measured
equivalence relation $(X,\mu,R)$ is determined by a measurable symmetric
subset $K\subset R\setminus\diag$. The result of the restriction of this graph
structure to an equivalence class $[x]$ gives the \emph{leafwise graph}
denoted by~$[x]^K$. We shall call $(X,\mu,R,K)$ a \emph{graphed equivalence
relation} \cite{Adams90}. Actually, in a somewhat less explicit form (in terms
of finitely generated pseudogroups) this definition is already present in
\cite{Plante75}, \cite{Series79} and \cite{Carriere-Ghys85}.

We shall always deal with the graph structures which are \emph{locally
finite}, i.e., any vertex has only finitely many neighbours, and denote by
$\deg$ the integer valued function which assigns to any point $x\in X$ the
degree (valency) of $x$ in the graph $[x]^K$. Passing, if necessary, to a
smaller equivalence relation, we may always assume that the graph structure is
\emph{leafwise connected}, i.e., a.e.\ leafwise graph $[x]^K$ is connected.
The latter condition means that
$$
R = \bigcup_{n\ge 1} K^n \;,
$$
with respect to the groupoid multiplication \eqref{eq:comp}.

The simplest example of a locally finite leafwise connected graph structure
arises in the situation when $R=R_G$ is the orbit equivalence relation of an
action of a finitely generated countable group $G$. For a symmetric generating
set $S$ put
$$
K = \{(x,y)\in R: y=sx \; \text{for a certain}\; s\in S\} \;.
$$
Then the leafwise graphs $[x]^K$ are isomorphic either to the (left)
\emph{Cayley graph} $(G,S)$ (if the orbit $Gx$ is free), or to the
\emph{Schreier graphs} determined by subgroups of $G$ (if the orbit $Gx$ is
not free). Once again, although any measured equivalence relation can be
generated by a group action, there is a lot of graph structures (for instance,
treed equivalence relations considered below) which can not be obtained in
this way (cf. the comment at the end of \secref{sec:equiv}).

If the measure $\mu$ is $R$-invariant and finite, then the leafwise graphs
$[x]^K$ have properties which make them similar to Cayley graphs of finitely
generated groups. In particular, under this condition $\deg\cdot\mu$ (the
measure $\mu$ multiplied by the density $\deg$) is a stationary measure of the
leafwise simple random walk along the classes of the graphed equivalence
relation $(X,\mu,R,K)$ \cite{Kaimanovich98}. Yet another property is related
to \emph{amenability} of the involved structures.

\subsection{Amenability of groups, graphs and equivalence relations}
\label{sec:amen}

There is a lot of definitions and applications of amenability, which
illustrates importance and naturalness of this notion. Here we shall just
briefly outline the properties which are used later on in this paper (see
\cite{Greenleaf69}, \cite{Anantharaman-Renault00} for the missing references
and for further details).

Let us first remind that the class of \emph{amenable groups} is, from the
analytical point of view, the most natural extension of the class of finite
groups. Indeed, finite groups can be characterized within the class of all (at
most) countable groups by existence of finite invariant measures. There are
two ways of ``extending'' the finiteness property to infinite groups. One can
look either for fixed points in a bigger space, or for approximative
invariance instead of precise one.

Von Neumann implemented the first idea and defined \emph{amenable groups} as
those which admit a translation invariant \emph{mean}, i.e, a finitely
additive probability measure (actually, the term ``amenable'' was introduced
much later by Day). Means being highly non-constructive objects, the other
option was explored by Reiter who introduced the following condition on a
countable group $G$: there exists an \emph{approximatively invariant sequence}
of probability measures $\theta_n$ on $G$, i.e., such that
$$
\|\theta_n - g \theta_n\| \toto_{n\to\infty} 0 \qquad\forall\,g\in G \;,
$$
where $\|\cdot\|$ denotes the total variation norm. Reiter proved that the
above condition (nowadays known as \emph{Reiter's condition}) is in fact
equivalent to amenability as defined by von Neumann.

By specializing Reiter's condition to sequences of probability measures
equidistributed on finite subsets of $G$ one obtains \emph{F{\o}lner's
condition}: there exists a sequence of finite subsets $A_n\subset G$ such that
$$
\frac{|gA_n \triangle A_n|}{|A_n|} \toto_{n\to\infty} 0 \qquad\forall\,g\in G
\;,
$$
where $\triangle$ denotes the symmetric difference of two sets, and $|A|$ is
the cardinality of a finite set $A$. This condition is also equivalent to
amenability of the group $G$.

For finitely generated groups the above approximative invariance condition on
a sequence of subsets $A_n\subset G$ takes especially simple form:
%$$
\begin{equation} \label{eq:iso}
\frac{|\pt A_n|}{|A_n|} \toto_{n\to\infty} 0 \;,
\end{equation}
%$$
where $\pt A$ denotes the \emph{boundary} of a set $A\subset G$ in the left
Cayley graph determined by a finite symmetric generating set (i.e., $\pt A$ is
the set of all points from $A$ which have a neighbour from the complement of
$A$). This is an \emph{isoperimetric} characterization of amenability. Its
formulation does not require any group structure, and therefore it can be
applied to arbitrary graphs. The graphs of \emph{bounded geometry} (i.e., with
uniformly bounded vertex degrees) which satisfy the above isoperimetric
property are called \emph{amenable}. In spectral terms amenable graphs are
characterized as the graphs for which the spectral radius of the Markov
operator of the simple random walk is 1, which is a generalization of Kesten's
description of amenable groups.

In a different direction the notion of amenability has been extended to group
actions, equivalence relations, and, more generally, to groupoids. Zimmer was
the first to notice that non-amenable groups may have actions which are
similar to actions of amenable groups. His original (rather heavy) definition
of amenable actions in terms of a fixed point property for Banach bundles was
almost immediately reformulated by Renault by using a modification of Reiter's
condition (although the work of Renault remained virtually unknown for quite a
while, cf. \cite{Kaimanovich97}). In particular, for discrete measured
equivalence relations this definition takes the following form: an equivalence
relation $(X,\mu,R)$ is \emph{amenable} if there exists a sequence of
measurable maps assigning to any point $x\in X$ a probability measure
$\theta_n^x$ on the equivalence class of $x$ such that
%$$
\begin{equation} \label{eq:am}
\|\theta_n^x-\theta_n^y\|\to 0 \qquad\text{for}\;\mu_\#\text{-a.e.}\; (x,y)\in
R \;.
\end{equation}
%$$

Thus, for a graphed equivalence relation $(X,\mu,R,K)$ there are two notions
of amenability. The ``global'' amenability is the amenability of the
equivalence relation $(X,\mu,R)$ in the sense of \eqref{eq:am} and does not
depend on the graph structure $K$, whereas the ``local'' or ``leafwise''
amenability means that $\mu$-a.e. graph $[x]^K$ is amenable in the sense of
\eqref{eq:iso}. In general these conditions are not equivalent (see
\cite{Kaimanovich97} for a complete description of their relationship).
However, for the purposes of the present paper we only need the following
implication: if the equivalence relation $(X,\mu,R)$ is amenable, the measure
$\mu$ is finite invariant, and the degrees of leafwise graphs of the structure
$K$ are uniformly bounded, then $\mu$-a.e.\ graph $[x]^K$ is also amenable
\cite{Carriere-Ghys85}.

\subsection{Random graphs and stochastic homogenization} \label{sec:gr}

A \emph{rooted} ($\equiv$ \emph{pointed}) graph $(\Ga,o)=\Ga_o$ is a graph
$\Ga$ endowed with a reference vertex $o\in\Ga$. Two rooted graphs $\Ga_o$ and
$\Ga'_{o'}$ are isomorphic ($\equiv$ isometric) if there is an isomorphism
($\equiv$ isometry with respect to the graph metric) $\f:\Ga\to\Ga'$ such that
$\f(o)=o'$. A graph $\Ga$ is \emph{rigid} if its isometry group $\Iso(\Ga)$ is
trivial; we shall also say that a rooted graph $\Ga_o$ is rigid if its
underlying graph $\Ga$ is rigid.

We shall denote by $\GG$ the space of (isometry classes of) infinite locally
finite connected rooted graphs, and by $\GG_\ni$ the subspace of $\GG$ which
consists of rigid rooted graphs. The space $\GG$ can be given a complete
separable metric by putting
$$
d(\Ga_o,\Ga'_{o'})=2^{-r} \;,
$$
where $r\ge 0$ is the maximal integer such that the $r$-balls centered at the
roots $o,o'$ of the graphs $\Ga,\Ga'$, respectively, are isometric as finite
rooted graphs. Thus, $\GG$ is a \emph{Polish space}, and therefore its Borel
structure is standard.

Given a graphed equivalence relation $(X,\mu,R,K)$, any point $x\in X$
determines the graph $[x]^K$. Let us denote by $[x]^K_\bullet=([x]^K,x)$ the
graph $[x]^K$ rooted at the point $x$. Thus, we have the map
$$
X \to \GG \;, \quad x \mapsto [x]^K_\bullet \;.
$$
In particular, if $\mu$ is a probability measure, then its image under the
above map is a probability measure on the space of rooted graphs $\GG$, i.e.,
a \emph{random rooted graph}.

Conversely, the space $\GG$ is endowed with a natural equivalence relation
$\R$: two rooted graphs $\Ga_o$ and $\Ga'_{o'}$ are equivalent if the
underlying graphs $\Ga$ and $\Ga'$ are isomorphic. It gives rise to a natural
graph structure $\K$ on $\R$ \cite{Kaimanovich98}:
$$
\K = \bigl\{ \bigl( \Ga_o, \Ga_{o'} \bigr): o\;\text{and}\;o'\;\text{are
neighbours in}\;\Ga \bigr\} \;.
$$
If the group of isometries of $\Ga$ is non-trivial, then the graph
$[\Ga_o]^\K$ is the quotient of the graph $\Ga$ with respect to the action of
the isometry group (in particular, it may contain loops). However, if
$\Iso(\Ga)$ is trivial, then $[\Ga_o]^\K$ is isomorphic to $\Ga$. Thus, the
restriction of the equivalence relation $\R$ to $\GG_\ni$ (which we shall also
denote by $\R$) has the following property:
$$
\ft{\textsf{the graph structure of the equivalence class of any rooted graph
$\Ga_o\in\GG_\ni$ is isomorphic to $\Ga$ itself.}}
$$

\begin{defn}[\cite{Kaimanovich03a}]
The random rooted graph determined by a probability measure~$\mu$ on the space
$\GG_\ni$ is \emph{stochastically homogeneous} if the measure $\mu$ is
invariant with respect to the equivalence relation $\R$.
\end{defn}

Below we shall give examples of stochastic homogenization of trees and their
horospheric products.

\section{Horospheric products of trees} \label{sec:2}

\subsection{Trees}

Recall that a \emph{tree} is a connected graph without cycles. Any two
vertices $x,y$ in a tree $T$ can be joined with a unique \emph{geodesic
segment} $[x,y]$. Any locally finite tree~$T$ has a natural
\emph{compactification} $\ov T = T \sqcup \pt T$ obtained in the following
way: a sequence of vertices $x_n$ which goes to infinity in $T$ converges in
this compactification if and only if for a certain ($\equiv$ any) reference
point $o\in T$ the geodesic segments $[o,x_n]$ converge pointwise. Thus, for
any reference point $o\in T$ the \emph{boundary} $\pt T$ can be identified
with the space of geodesic rays issued from $o$ (and endowed with the topology
of pointwise convergence). There are many other equivalent descriptions of the
boundary $\pt T$ (and of the compactification~$\ov T$), in particular, as the
\emph{space of ends} of $T$ and as the \emph{hyperbolic boundary} of $T$.

A tree $T$ with a distinguished boundary point $\g\in\pt T$ is called
\emph{pointed at infinity} ($\equiv$ remotely rooted; in the terminology of
Cartier \cite{Cartier72} the point $\g$ is called a ``mythological
progenitor''). A triple $T_o^\g=(T,o,\g)$ with $o\in T$ and $\g\in\pt T$ is a
\emph{rooted tree pointed at infinity}.

Any two geodesic rays converging to the same boundary point eventually
coincide, so that any boundary point $\g\in\pt T$ determines the associated
additive $\Z$-valued \emph{Busemann cocycle} on $T\times T$. It is defined as
%$$
\begin{equation} \label{eq:Bu}
\b_\g(x,y) = d(y,o) - d(x,o) \;,
\end{equation}
%$$
where $d$ is the graph distance on $T$, and $o$ is the \emph{confluence} of
the geodesic rays $[x,\g)$ and $[y,\g)$, see \figref{fig:bus}.

\begin{figure}[h]
\begin{center}
     \psfrag{x}[cl][cl]{$x$}
     \psfrag{o}[cl][cl]{$o$}
     \psfrag{y}[cl][cl]{$y$}
     \psfrag{g}[cl][cl]{$\g$}
     \psfrag{d}[cl][cl]{$\pt T$}
          \includegraphics[scale=.6]{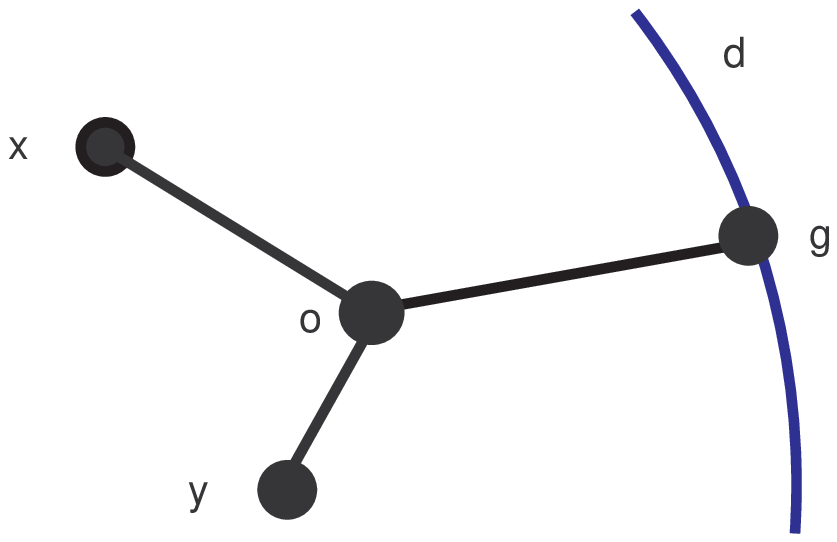}
          \end{center}
          \caption{}
          \label{fig:bus}
\end{figure}

The Busemann cocycle can also be defined as
$$
\b_\g(x,y) = \lim_{z\to\g} \bigl[ d(y,z) - d(x,z) \bigr] \;,
$$
so that it is a ``regularization'' of the formal expression $d(y,\g)-d(x,\g)$.
In the presence of a reference point $o\in T$ one can also talk about the
\emph{Busemann function}
$$
b_\g(x) = \b_\g(o,x) \;.
$$
The level sets
$$
H_k = \{x\in T : b_\g(x) = k \}
$$
of the Busemann function ($\equiv$ of the Busemann cocycle) are called
\emph{horospheres} centered at the boundary point $\g$, see \figref{fig:tree}.

\begin{figure}[h]
\begin{center}
     \psfrag{hm}[cl][cl]{$H_{-1}$}
     \psfrag{o}[cl][cl]{$o$}
     \psfrag{h}[cl][cl]{$H_0$}
     \psfrag{g}[cl][cl]{$\g$}
     \psfrag{hp}[cl][cl]{$H_1$}
          \includegraphics[scale=.6]{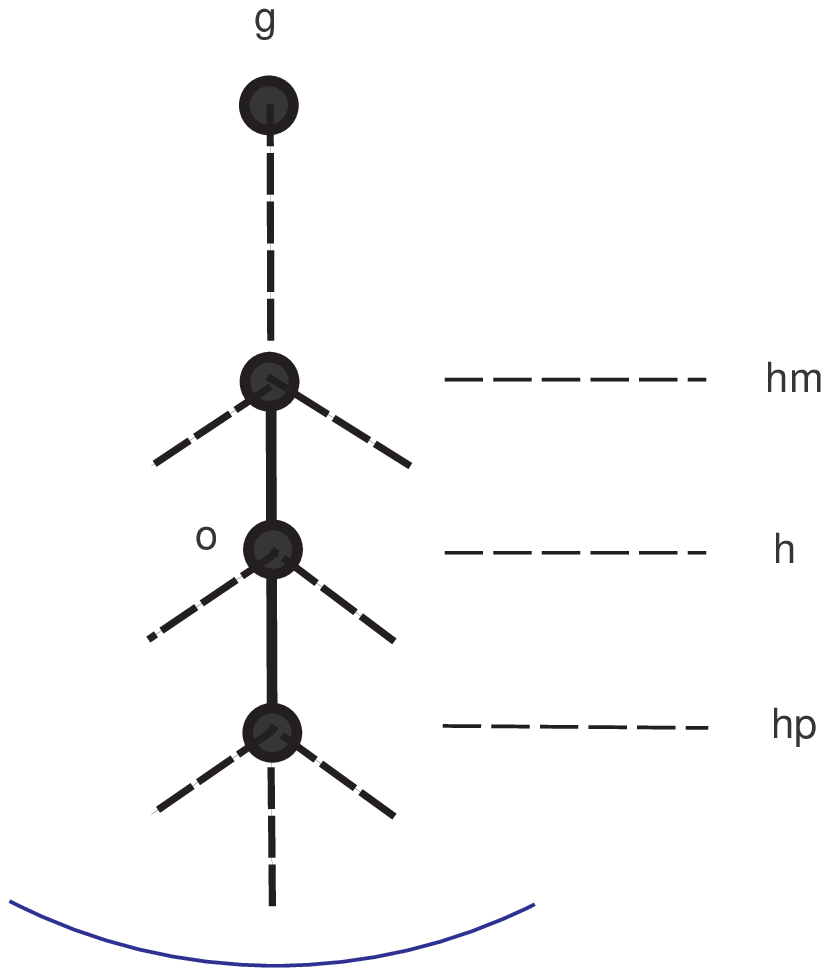}
          \end{center}
          \caption{}
          \label{fig:tree}
\end{figure}

\subsection{Horospheric products}

\begin{defn} \label{def:DL}
Let $T=(T,o,\g)$ and $T'=(T',o',\g')$ be two rooted trees pointed at infinity,
and let $b=\b_{\g}(o,\cdot),\; b'=\b_{\g'}(o',\cdot)$ be the corresponding
Busemann functions. The \emph{horospheric product} $T\du T'$ is the graph with
the vertex set
$$
\{(x,x')\in T\times T': b(x)+b'(x')=0 \}
$$
and the edge set
$$
\bigl\{\bigl( (x,x'), (y,y') \bigr) : (x,y) \;\text{and}\;(x',y') \;\text{are
edges in}\; T,T',\;\text{respectively} \bigr\} \;.
$$
\end{defn}

\begin{rem}
In the cocycle language, the product $T\times T'$ is endowed with the
$\Z$-valued additive cocycle $c=\b_{\g}+\b_{\g'}$. Its \emph{kernel} $\ker
c=c^{-1}(0)\subset T\times T'$ consists of connected components (with respect
to the product graph structure) which are the horospheric products
corresponding to different choices of the roots $o,o'$.
\end{rem}

Geometrically one can think about the horospheric products in the following
way \cite{Kaimanovich-Woess02}. Draw the tree $T'$ upside down next to $T$ so
that the respective horospheres $H_k(T)$ and $H_{-k}(T')$ are at the same
level. Connect the two origins $o,o'$ with an elastic spring. It can move
along each of the two trees, may expand infinitely, but must always remain in
horizontal position. The vertex set of $T\du T'$ consists then of all
admissible positions of the spring. From a position $(x,x')$ with
$b(x)+b'(x')=0$ the spring may move downwards to one of the ``sons'' of $x$
and at the same time to the ``father'' of $x'$, or upwards in an analogous
way. Such a move corresponds to going to a neighbour $(y,y')$ of $(x,x')$, see
\figref{fig:dl}.

\begin{figure}[h]
\begin{center}
     \psfrag{o2}[cl][cl]{$o'$}
     \psfrag{o}[cl][cl]{$o$}
     \psfrag{g2}[cl][cl]{$\g'$}
     \psfrag{g}[cl][cl]{$\g$}
          \includegraphics[scale=.6]{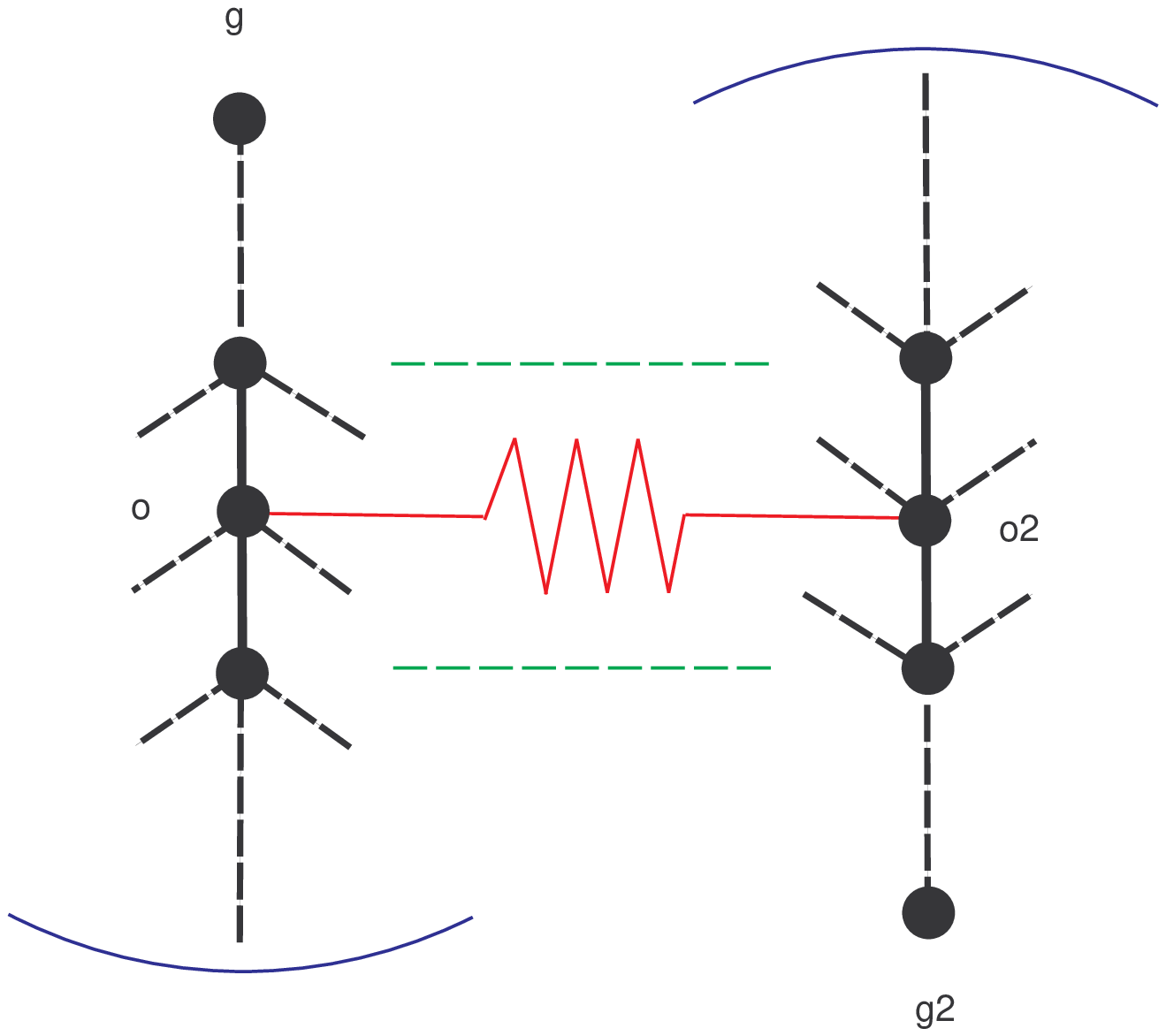}
          \end{center}
          \caption{}
          \label{fig:dl}
\end{figure}

\begin{rem}
This construction (in a different terminology) was first introduced by Diestel
and Leader \cite{Diestel-Leader01} in an attempt to answer a question of Woess
\cite[Problem~1]{Woess91}: \emph{is there a locally finite vertex-transitive
graph which is not quasi-isometric with a Cayley graph of some finitely
generated group?} They suggested the product $\Tb_3\du\Tb_4$ of homogeneous
trees of degrees 3 and 4, respectively, as a possible example. That this is
indeed so was only recently proved by Eskin, Fisher and Whyte
\cite{Eskin-Fisher-Whyte07}. In the meantime the construction of Diestel and
Leader attracted a lot of attention because of its numerous interesting
features (see \cite{Woess05,Bartholdi-Neuhauser-Woess08} and the references
therein). The arising graphs are also known under the names of
\emph{Diestel--Leader graphs} (when the multipliers are homogeneous trees) and
of \emph{horocyclic products} (however, we feel that the adjective
``horospheric'' is more appropriate, because the level sets of Busemann
functions in trees are anything but cycles).
\end{rem}

\subsection{Treed equivalence relations}

A graphed equivalence relation $(X,\mu,R,K)$ is \emph{treed} if a.e.\ leafwise
graph is a tree \cite{Adams90}. We shall denote by $T_x=[x]^K_\bullet$ the
leafwise tree of a point $x$ rooted at $x$. The \emph{boundary bundle} of a
treed equivalence relation $(X,\mu,R,K)$ is the set
$$
\wt X = \{(x,\g): x\in X,\g\in\pt T_x\} = \bigcup_{x\in X} \{x\}\times\pt
T_x\;,
$$
so that it is an analogue of the unit tangent bundle of a negatively curved
manifold (since we are talking about equivalence relations, a better analogue
is actually the unit tangent bundle of a foliation with negatively curved
leaves, like, for instance, the stable foliation of the geodesic flow on a
negatively curved manifold, see \remref{rem:geod} below).

We shall endow $\wt X$ with the equivalence relation (the \emph{boundary
extension} of $R$)
$$
\begin{aligned}
\wt R &= \bigl\{ \bigl( (x,\g),(y,\g) \bigr): (x,y)\in R, \;\g\in\pt
T_x\cong\pt T_y \bigr\} \\
&\cong \{ (x,y,\g): (x,y)\in R, \;\g\in\pt T_x\cong\pt T_y \}
\end{aligned}
$$
(if two points $x,y\in X$ are equivalent, then $T_x$ and $T_y$ coincide as
unrooted trees, so that there is a natural identification of the boundaries
$\pt T_x$ and $\pt T_y$) and with the treed graph structure $\wt K$ inherited
from $X$.

The equivalence relation $\wt R$ is endowed with the $\Z$-valued additive
cocycle (which we shall also call \emph{Busemann} along with the cocycle
\eqref{eq:Bu})
%$$
\begin{equation} \label{eq:B}
\wt\b: (x,y,\g) \mapsto \b_\g(x,y) \;.
\end{equation}
%$$
It will play an important role in the sequel.

In order to endow $\wt X$ with a Borel structure, we shall fix, once and
forever, a Borel identification of the space $X$ with the unit interval. This
identification provides us with a linear order on any subset of $X$. In
particular, the set of neighbours of any point $x\in X$ can be canonically
identified with the set $\{1,2,\dots,d\}$, where $d=\deg x$. Thus, in the case
of a treed equivalence relation we can record any leafwise geodesic
$x=x_0,x_1,x_2,\dots$ issued from a point $x\in X$ as a sequence
$n_1,n_2,\dots$, where $n_i$ is the position of $x_i$ among the neighbours of
$x_{i-1}$. In this way we obtain, for any $x\in X$, a one-to-one map $\tau_x$
from $\pt X$ to a Borel subset of $\N^\N$ (this is similar to the well-known
\emph{Ulam--Harris notation}). Note that the maps $\tau_x$ \emph{do} depend on
$x$ (not only on its equivalence class!), although the boundaries $\pt T_x,\pt
T_y$ can be identified for any two equivalent points $x,y\in X$.

Finally, let us introduce a measure on the boundary extension $\wt X$ which
would be quasi-invariant with respect to the equivalence relation $\wt R$.
Since $\wt X$ is fibered over $X$, it is natural to construct such a measure
by integrating a system of measures on the fibers against the measure $\mu$ on
the base.

\begin{defn} \label{def:ext}
Given a treed equivalence relation $(X,\mu,R,K)$, a system of finite measures
$\{\nu_x\}_{x\in X}$ on the boundaries $\pt T_x$ of leafwise trees $T_x$ is
\emph{measurable} if the map
$$
x\mapsto\tau_x(\nu_x)
$$
from $X$ to the space of measures on $\N^\N$ is (weakly) measurable (i.e., for
any measurable function $f$ on $\N^\N$ the integrals $\langle
f,\tau_x(\nu_x)\rangle$ depend on $x$ measurably). A measurable system of
boundary measures $\{\nu_x\}$ is \emph{quasi-invariant} if for $\mu_\#$-a.e.\
pair $(x,y)\in R$ the measures $\nu_x$ and $\nu_y$ are equivalent. A
measurable system of boundary measures $\{\nu_x\}$ gives rise to the measure
$$
d\wt\mu(x,\g) = d\mu(x)d\nu_x(\g)
$$
on the boundary bundle $\wt X$ which is called a \emph{boundary extension} of
$\mu$.
\end{defn}

\begin{rem}
Obviously,
$$
\|\wt\mu\| = \int \|\nu_x\| d\mu(x) \;,
$$
and the measure $\wt\mu$ is finite if and only if the above integral is
finite.
\end{rem}

\begin{rem}
Another definition of a measurable boundary system of measures over a graphed
equivalence relation with hyperbolic leaves is given in \cite{Kaimanovich04}
(in terms of separable measurable bundles of Banach spaces). One can establish
the equivalence of these two definitions for treed equivalence relations in a
rather straightforward (if tedious) way.
\end{rem}

\begin{prop}
Let $(X,\mu,R,K)$ be a treed equivalence relation. Quasi-invariance of a
measurable system of boundary measures $\{\nu_x\}$ is equivalent to
quasi-invariance of the measure $\wt\mu$ with respect to the equivalence
relation $\wt R$, and the Radon--Nikodym cocycle $\wt\D$ of the measure
$\wt\mu$ with respect to $\wt R$ is connected with the Radon--Nikodym cocycle
$\D$ of the measure $\mu$ with respect to the equivalence relation $R$ by the
formula
$$
\wt\D (x,y,\g) = \D(x,y) \frac{d\nu_y}{d\nu_x}(\g) \;.
$$
\end{prop}

\begin{proof}
In the language of ``differentials'' (see \secref{sec:RN}) of the involved
measures
$$
\wt\D (x,y,\g) = \frac{d\wt\mu(y,\g)}{d\wt\mu(x,\g)} =
\frac{d\mu(y)d\nu_y(\g)}{d\mu(x)d\nu_x(\g)} = \frac{d\mu(y)}{d\mu(x)} \cdot
\frac{d\nu_y(\g)}{d\nu_x(\g)} = \D(x,y) \frac{d\nu_y(\g)}{d\nu_x(\g)} \;.
$$
\end{proof}

\begin{cor} \label{cor:RN}
If the measure $\mu$ is $R$-invariant, then the Radon--Nikodym cocycle of the
measure $\wt\mu$ with respect to the equivalence relation $\wt R$ coincides
with the pairwise Radon--Nikodym derivatives of the boundary system of
measures $\{\nu_x\}$:
$$
\wt\D (x,y,\g) = \frac{d\nu_y}{d\nu_x}(\g) \;.
$$
\end{cor}

\begin{rem} \label{rem:geod}
For negatively curved manifolds similar boundary extensions (where the
boundary is the visibility sphere of the universal covering manifold)
naturally arise in the study of invariant measures of the geodesic flow (e.g.,
see \cite{Kaimanovich90,Ledrappier95}). Yet another boundary extension of a
different kind can be associated with the measure-theoretical Poisson
boundaries rather than with the topological ones (see \cite{Kaimanovich95}).
\end{rem}

\begin{defn} \label{def:conf}
A measurable system of boundary measures $\{\nu_x\}$ over a treed equivalence
relation $(X,\mu,R,K)$ is called \emph{conformal of dimension $\l>0$} if it is
quasi-invariant and its Radon--Nikodym derivatives satisfy the relation
%$$
\begin{equation} \label{eq:RN}
\frac{d\nu_y}{d\nu_x}(\g) = e^{-\l\b_\g(x,y)}
\quad\text{for}\;\wt\mu_\#\text{-a.e.}\;(x,y,\g)\in\wt R\;.
\end{equation}
%$$
\end{defn}

\begin{rem}
This definition is analogous to the definition of conformal streams (measures,
densities, see \cite{Patterson76,Sullivan79,Kaimanovich-Lyubich05}) for
negatively curved manifolds. Note that in the group case the identity
\eqref{eq:RN} only makes sense in combination with the requirement of
equivariance of the map $x\mapsto\nu_x$ (as otherwise one gets a conformal
system from an arbitrary measure by multiplying it by the exponent of the
Busemann cocycle). In our situation the equivariance condition is replaced
with the requirement that the system of measures $\{\nu_x\}$ be measurable (in
perfect agreement with the general spirit of the theory of equivalence
relations which consists in replacing group invariance with measurablity).
\end{rem}

\begin{rem}
If the measure $\mu$ is invariant, then, in view of \corref{cor:RN}, a system
$\{\nu_x\}$ is conformal if and only if the logarithm of the Radon--Nikodym
cocycle of the measure~$\wt\mu$ with respect to the equivalence relation $\wt
R$ is proportional to the Busemann cocycle on $\wt R$ with the proportionality
coefficient $-\l$.
\end{rem}

\begin{rem}
If the measure $\mu$ is finite invariant, and $\mu$-a.e.\ tree $T_x$ has at
least 3 ends, then in fact a.e.\ tree has a continuum of ends and the
equivalence relation $(X,R,\mu)$ is non-amenable \cite{Adams90}, which implies
that in this situation there are no invariant measurable systems of boundary
measures (i.e., there are no conformal systems of dimension~$0$)
\cite{Kaimanovich04}.
\end{rem}

\subsection{Horospheric product of treed equivalence relations}

Let us first remind that the \emph{product} of two equivalence relations
$(X,R)$ and $(X',R')$ is the equivalence relation
$$
R\times R' = \bigl\{\bigl( (x,x'),(y,y') \bigr): (x,y)\in R, (x',y')\in R'
\bigr\}
$$
on the state space $X\times X'$. If the relations $R,R'$ are endowed with the
respective graph structures $K,K'$, then the product relation carries the
natural product graph structure $K\times K'$ (an edge in the product is the
product of edges in the multipliers). Finally, if $\mu$ (resp., $\mu'$) is a
$R$- (resp., $R'$-) quasi-invariant measure on $X$ (resp., $X'$) with the
Radon--Nikodym cocycle $\D$ (resp., $\D'$), then the product measure
$\mu\times\mu'$ is $R\times R'$-quasi-invariant, and its Radon--Nikodym
cocycle is $\D\times\D'$.

Let now $(X,R,K),(X',R',K')$ be two treed equivalence relations, and let $(\wt
X,\wt R,\wt K)$ and $(\wt X',\wt R',\wt K')$ be their respective boundary
extensions endowed with the Busemann cocycles $\wt\b,\wt\b'$ \eqref{eq:B}, so
that their product $(\wt X\times\wt X',\wt R\times\wt R',\wt K\times\wt K')$
carries the cocycle $c=\wt\b+\wt\b'$.

\begin{defn}
The \emph{horospheric product} of treed equivalence relations $(X,R,K)$ and
$(X',R',K')$ is the equivalence relation $\RR=\ker c\subset\wt R\times\wt R'$
on the product $\wt X\times\wt X'$. It is endowed with the graph structure
$\KK = \RR \cap \wt K\times\wt K'$.
\end{defn}

Thus, the $\RR$-equivalence class of a point $(x,\g,x',\g')\in\wt X\times\wt
X'$ endowed with the graph structure $\KK$ is precisely the horospheric
product of the pointed at infinity rooted trees $([x]^K_\bullet,\g)$ and
$([x']^{K'}_\bullet,\g')$ in the sense of \defref{def:DL}.

Let us now endow the treed equivalence relations $(X,R,K)$ and $(X',R',K')$
with respective quasi-invariant measures $\mu,\mu'$, let $\{\nu_x\},
\{\nu'_{x'}\}$ be respective quasi-invariant measurable systems of boundary
measures, and let $\wt\mu,\wt\mu'$ be the corresponding quasi-invariant
measures for the boundary extensions $(\wt X,\wt R,\wt K)$ and $(\wt X',\wt
R',\wt K')$. Then the product measure $\wt\mu\times\wt\mu'$ is
$\RR$-quasi-invariant (since $\RR\subset \wt R\times\wt R'$), and its
Radon--Nikodym cocycle $\DD$ is the restriction of the Radon--Nikodym cocycle
of the measure $\wt\mu\times\wt\mu'$ from $\wt R\times\wt R'$ to $\RR$. Thus,
we obtain

\begin{thm} \label{th:hp}
Let $(X,\mu,R,K)$ and $(X',\mu',R',K')$ be treed equivalence relations with
finite invariant measures, and let $\{\nu_x\},\{\nu'_{x'}\}$ be respective
conformal measurable systems of boundary measures of the same dimension
$\l>0$. Then the resulting measure $\wt\mu\times\wt\mu'$ on the horospheric
product of these treed equivalence relations $(\wt X\times\wt
X',\wt\mu\times\wt\mu',\RR,\KK)$ is $\RR$-invariant.
\end{thm}

\begin{proof}
By \corref{cor:RN} and \defref{def:conf} the Radon--Nikodym cocycle of the
measure $\wt\mu\times\wt\mu'$ with respect to the equivalence relation $\wt
R\times\wt R'$ is
$$
e^{-\l (\wt\b + \wt\b')} = e^{- \l c} \;,
$$
where $c=\wt\b+\wt\b'$ is precisely the cocycle whose kernel determines the
horospheric product of our treed equivalence relations.
\end{proof}

\begin{rem}
Our construction of an invariant measure from two conformal systems of
boundary measures of the same dimension is based on the same idea as the
construction of an invariant measure of the geodesic flow on a negatively
curved manifold from a single conformal measure (see \cite{Kaimanovich90} and
the Appendix in \cite{Kaimanovich-Lyubich05}). The only difference is that in
our situation we deal with two boundary systems rather than one in the
classical case.
\end{rem}

\begin{thm} \label{th:am}
Let $(X,\mu,R,K)$ and $(X',\mu',R',K')$ be treed equivalence relations with
finite invariant measures, and let $\{\nu_x\},\{\nu'_{x'}\}$ be respective
conformal measurable systems of boundary measures of the same dimension
$\l>0$. If the measures $\wt\mu,\wt\mu'$ are both finite, and if the valencies
of the structures $K,K'$ are uniformly bounded, then for
$\wt\mu\times\wt\mu'$-a.e.\ $(x,\g,x',\g')$ the horospheric product of the
pointed at infinity rooted trees $([x]^K_\bullet,\g)$ and
$([x']^{K'}_\bullet,\g')$ is amenable.
\end{thm}

\begin{proof}
The treed equivalence relations $(\wt X,\wt\mu,\wt R,\wt K)$ and $(\wt
X',\wt\mu',\wt R',\wt K')$ are both pointed at infinity, and therefore
amenable \cite{Kaimanovich04}. Thus, their product is also amenable together
with the subrelation $\RR$ (this follows, for instance, from the description
of amenable equivalence relations as the ones which are orbit equivalent to
$\Z$-actions \cite{Connes-Feldman-Weiss81}). On the other hand, since the
measure $\wt\mu\times\wt\mu'$ is finite and $\RR$-invariant, amenability of
$\RR$ implies amenability of a.e.\ associated leafwise graph (see
\secref{sec:amen}).
\end{proof}

\section{Galton--Watson trees} \label{sec:3}

In this Section we shall discuss an example of invariant measures on treed
equivalence relations and the associated horospheric products arising from
branching processes.

\subsection{Augmented process}

Let $p=\{p_k\}$ be a probability distribution on the set
$\Z_+=\{0,1,2,\dots\}$. It gives rise to the random rooted tree $\TT_\oo$,
which is the ``genealogical tree'' of the associated \emph{Galton--Watson
branching process}: the number of offspring of the progenitor $\oo$ (the root
of the tree) is distributed according to the law $p$, each of them also
produces its own offspring according to the same law and independently of all
the rest, etc. (see \figref{fig:gw}). We shall denote by $\P=\P(p)$ the
corresponding probability measure on the space of locally finite rooted trees
$\T\subset\GG$.

\begin{figure}[h]
\begin{center}
          \psfrag{o}[cl][cl]{$\oo$}
          \includegraphics[scale=.6]{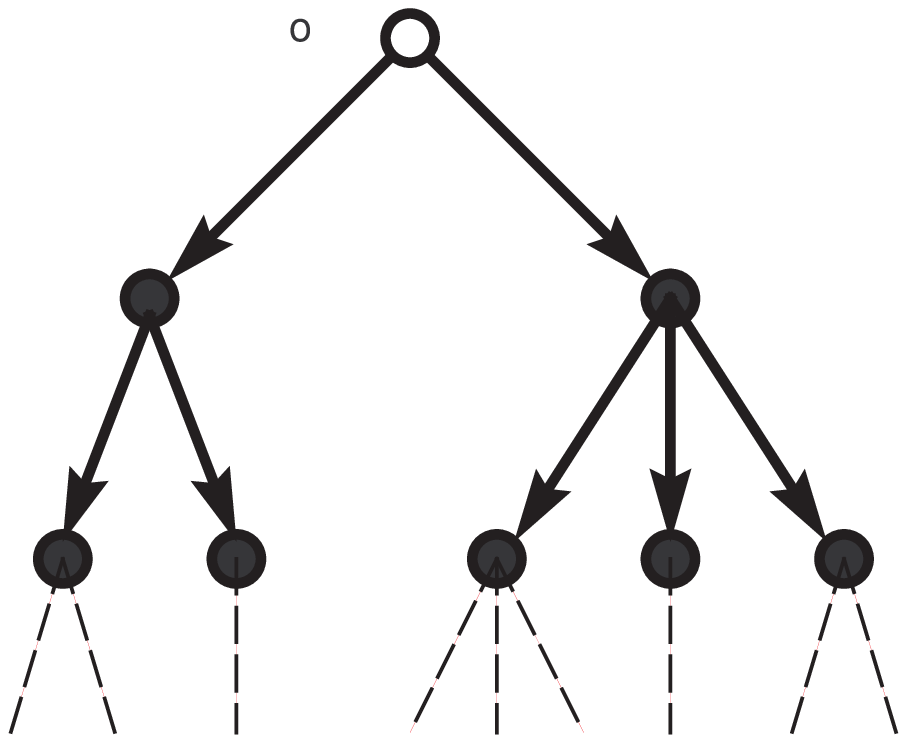}
          \end{center}
          \caption{}
          \label{fig:gw}
\end{figure}

For simplicity we shall assume that
%$$
\begin{equation} \label{eq:stand}
\ft{\textsf{(i)\quad\; $p_0=0$, so that the extinction probability is 0, and
the tree $\TT_\oo$ is a.s.\ infinite and has no leaves; \newline (ii) \quad
the support of the distribution $p$ contains more than one point, so that the
tree $\TT_\oo$ is a.s.\ rigid and has a continuum of ends.}}
\end{equation}
%$$

The measure $\P$ on the space $\T_\ni\subset\GG_\ni$ of rooted rigid trees is
\emph{not} quasi-invariant with respect to the natural equivalence relation
$\R$ (or, rather, its restriction to $\T_\ni$ which we also denote by $\R$,
see \secref{sec:gr}). The reason for this is the fact that the root $\oo$ is
different from other vertices of $\TT$, because statistically it has one
neighbour less (we skip the rigorous argument). However, a little modification
of the Galton--Watson process (which we describe below) provides an
$\R$-invariant measure on $\T_\ni$.

The \emph{augmented Galton--Watson process} introduced in
\cite{Lyons-Pemantle-Peres95} is defined in the same way as the original
Galton--Watson process with the only difference that the number of offspring
of the progenitor (only) has the distribution $p'_k=p_{k-1}$ (i.e., the root
has $k+1$ offspring with probability $p_k$), and these offspring all have
independent standard Galton--Watson descendant trees with offspring
distribution $\{p_k\}$. In other words, the number of offspring of the
progenitor is ``by force'' increased by one. Denote by $\P'=\P'(p)$ the
associated probability measure on the space of rigid rooted trees $\T_\ni$.

\subsection{Invariant measure}

\begin{thm} \label{th:inv}
The measure $\displaystyle\mu=\frac{1}{\deg}\P'$ on $\T_\ni$ is
$\R$-invariant.
\end{thm}

This theorem was proved in \cite{Kaimanovich98} by using the fact that $\P'$
is a stationary measure of the leafwise simple random walk on $(\T_\ni,\R)$
\cite{Lyons-Pemantle-Peres95} and a relation between stationary and invariant
measures established in \cite{Kaimanovich98} (cf. \secref{sec:str}). For the
sake of completeness we shall give a simple direct proof of \thmref{th:inv}.

\begin{proof}[Proof of \thmref{th:inv}]
Let
$$
\R_1 = \{(x,y)\in\R: d(x,y)=1\} \subset\T_\ni\times\T_\ni \;,
$$
where $d$ is the graph metric of the canonical graph structure on the
equivalence relation~$\R$. One can think about $\R_1$ as the set of
\emph{doubly rooted rigid trees}; its elements are triples $(T,o,o')$, where
$T$ is a rigid tree, $o\in T$ is its \emph{principal root}, and its
\emph{secondary root} $o'\in T$ is at distance 1 from $o$. Denote by
$\P_{\!\!\!\leftrightarrow}$ the probability measure on $\R_1$ obtained in the
following way: consider the principal and the secondary roots as the
progenitors of two independent Galton--Watson trees with the distribution $p$,
and then join these roots with an edge, see \figref{fig:gw2}.

\begin{figure}[h]
\begin{center}
          \psfrag{o}[cl][cl]{$o$}
          \psfrag{o1}[cl][cl]{$o'$}
          \includegraphics[scale=.6]{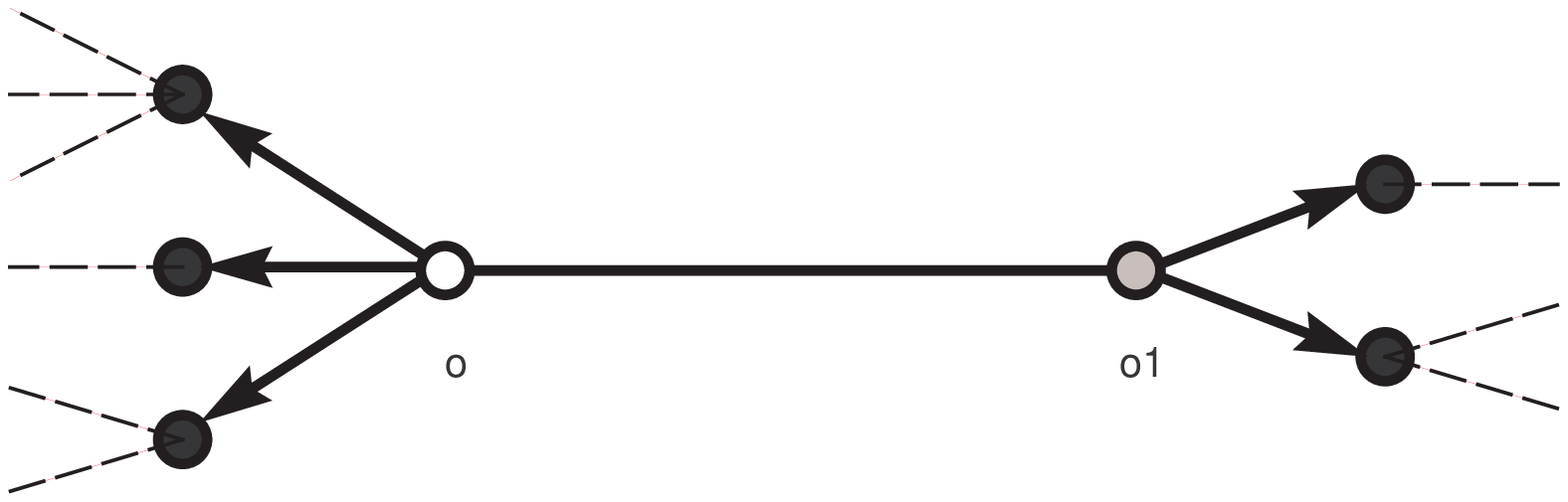}
          \end{center}
          \caption{}
          \label{fig:gw2}
\end{figure}

The measure $\P_{\!\!\!\leftrightarrow}$ coincides with the restriction to
$\R_1$ of the left counting measure $\mu_\#$ associated with the measure
$\mu$. Since $\P_{\!\!\!\leftrightarrow}$ is obviously invariant with respect
to the involution~$\th$ (which consists in exchanging the principal and the
secondary roots, see \figref{fig:flip}),
%$$
\begin{equation} \label{eq:sec}
\ft{\textsf{any partial transformation of the equivalence relation $\R$ whose
graph is contained in $\R_1$ preserves the measure $\mu$.}}
\end{equation}
%$$
This property easily implies that the measure $\mu$ is preserved by \emph{all}
partial transformations of $\R$, i.e, is $\R$-invariant. Indeed, let $A$ be a
$\mu$-negligible subset of $\T_\ni$. Then by \eqref{eq:sec} its
1-neighbourhood (with respect to the leafwise graph distance) is
$\mu$-negligible as well, and so on, so that the $\R$-saturation of $A$ is
also $\mu$-negligible. Thus, $\mu$ is $\R$-quasi-invariant. By \eqref{eq:sec}
its Radon--Nikodym cocycle $\D$ is identically 1 on $\R_1$; therefore by the
cocycle identity $\D\equiv 1$ on $\R$.
\end{proof}

\begin{figure}[h]
\begin{center}
          \psfrag{T}[cl][cl]{$\th$}
          \includegraphics[scale=.6]{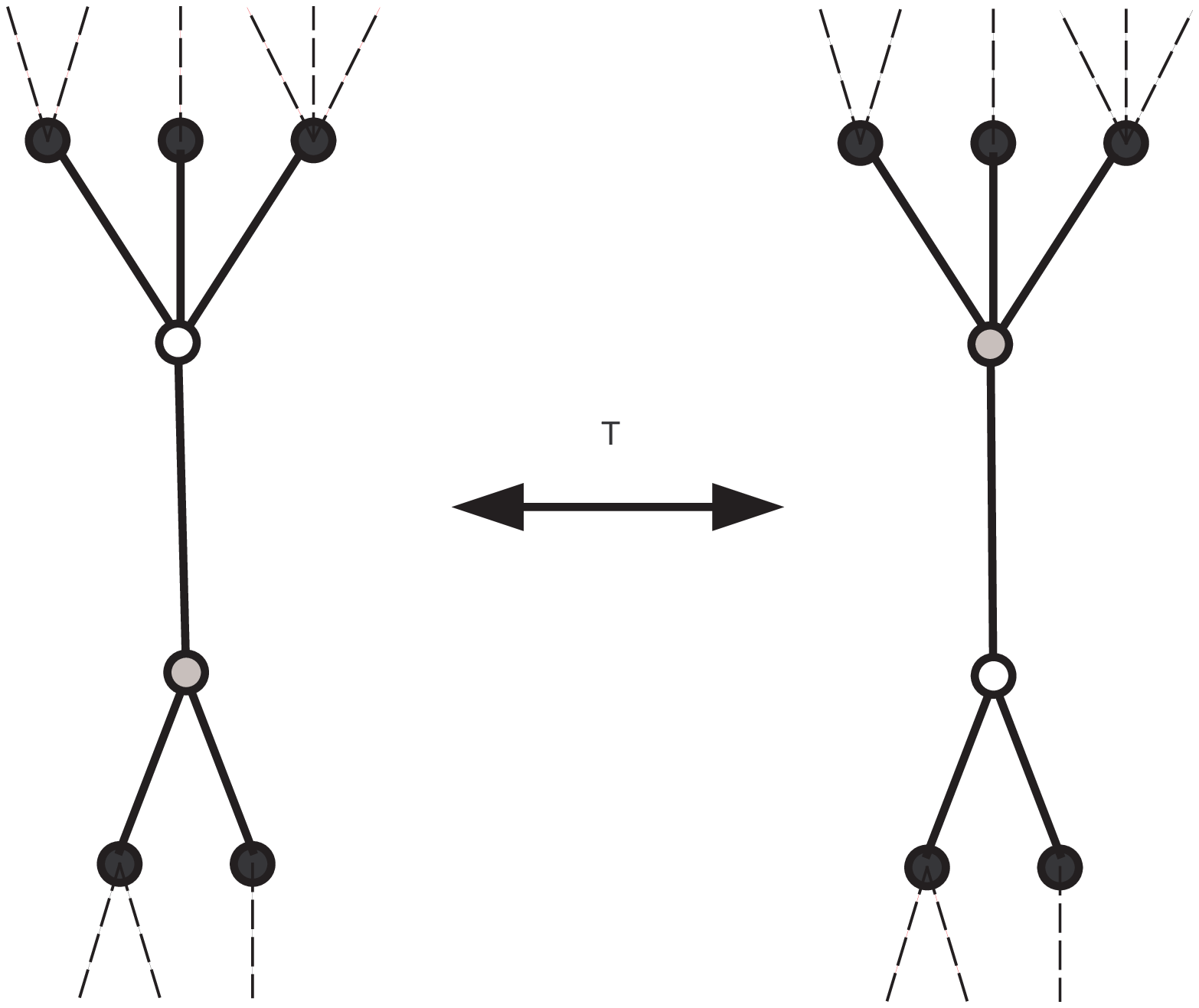}
          \end{center}
          \caption{}
          \label{fig:flip}
\end{figure}

\subsection{Boundary measure}

Given a rooted tree $T_o$ denote by $S^n_o\subset T$ the $n$-sphere centered
at the root $o$. If the distribution $p$ has a finite first moment
$$
m = \sum k \,p_k \;,
$$
then by \eqref{eq:stand} $m>1$ (the associated branching process is
\emph{supercritical}), and, as it was noticed already by Doob (this was one of
the first applications of the martingale theory), for $\P$-a.e.\
Galton--Watson tree $T_o$ there exists a limit
%$$
\begin{equation} \label{eq:L}
\lim_{n\to\infty}\frac{|S^n_o|}{m^n} = L\;.
\end{equation}
%$$
Earlier works containing sufficient conditions for a.s.\ positivity of the
limit \eqref{eq:L} \cite{Harris48,Levinson59} culminated in the following

\begin{thm}[Kesten--Stigum
 \cite{Kesten-Stigum66,Athreya-Ney72,Lyons-Pemantle-Peres95a}] \label{th:KS}
Under the assumption $p_0=0$ either
\begin{itemize}
    \item[(i)]
    $\sum k \log k \,p_k < \infty$,
    \item[(ii)]
    $L>0$  $\P$-a.s.,
    \item[(iii)]
    $\E L = 1$, where $\E$ denotes the expectation with respect to the measure
    $\P$;
\end{itemize}
or
\begin{itemize}
    \item[(i$'$)]
    $\sum k \log k\, p_k = \infty$,
    \item[(ii$'$)]
    $L=0$  $\P$-a.s.
\end{itemize}
\end{thm}

The idea that existence and positivity of the limit \eqref{eq:L} can be used
in order to define a measure on the boundary of the Galton--Watson tree is
very natural, and apparently for the first time appeared in \cite{Holmes73}.
Nowadays this boundary measure is usually known under the name of the
\emph{branching measure} (for instance, see \cite{Liu2001} and the references
therein). We shall need this result in a slightly modified form: for the
augmented Galton--Watson trees instead of the usual ones.

\begin{thm} \label{th:exist}
Denote by $\Nu^n_o$ the counting measure on the $n$-sphere $S^n_o$ of a rooted
tree~$T_o$. If the distribution $p=(p_k)$ satisfies condition \textup{(i)}
from \thmref{th:KS} then the limit measure (with respect to the weak$^*$
topology on the compactification $\ov T$)
%$$
\begin{equation} \label{eq:lim}
\nu=\lim_n \frac{\Nu^n_o}{m^n}
\end{equation}
%$$
exists for $\P'$-a.e.\ tree $T_o$, and the expectation of its norm is
%$$
\begin{equation} \label{eq:1m}
\E' \|\nu\| = 1+\frac1m \;.
\end{equation}
%$$
\end{thm}

\begin{proof}
For a point $x\in T\setminus\{o\}$ denote by $\S_o^x\subset\pt T$ the
\emph{shadow} of the point $x$ as viewed from the root $o$, i.e., the set of
endpoints of all geodesic rays issued from $o$ and passing through $x$, see
\figref{fig:shadow}.

\begin{figure}[h]
\begin{center}
     \psfrag{x}[cl][cl]{$o$}
     \psfrag{y}[cl][cl]{$x$}
     \psfrag{S}[cl][cl]{$\S_o^x$}
     \psfrag{p}[cl][cl]{$\pt T$}
          \includegraphics[scale=.6]{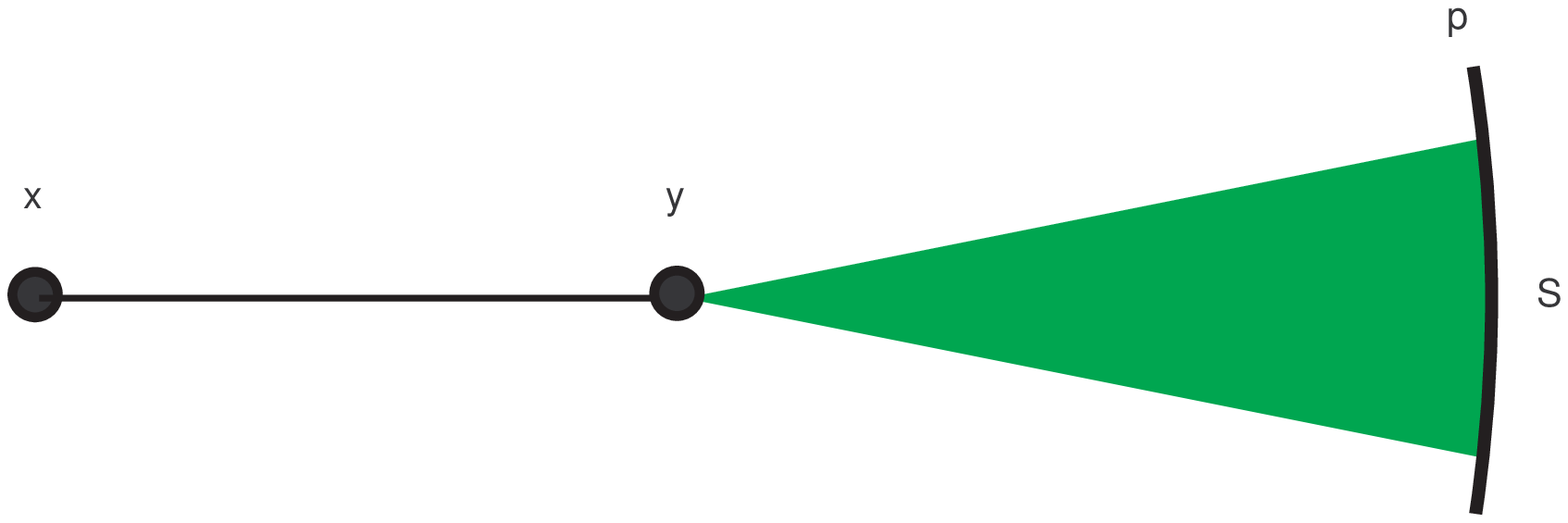}
          \end{center}
          \caption{}
          \label{fig:shadow}
\end{figure}

Weak convergence of the sequence \eqref{eq:lim} is equivalent to convergence
of the sequences
$$
\frac{\Nu^n_o(\S_o^x)}{m^n} = \frac{|S^n_o\cap\S_o^x|}{m^n}
$$
for all $x\in T$, which follows from \eqref{eq:L}, because $\S_o^x\cap T$ is a
$\P$-distributed Galton--Watson tree growing from the root $x$. Formula
\eqref{eq:1m} is then the result of ``rescaling'' property (iii) from
\thmref{th:KS}: the expected number of offspring of the root is $m+1$, whereas
for each of them the expected mass of the boundary measure is $1/m$.
\end{proof}

Now we can endow the treed equivalence relation $(\T_\ni,\P',\R,\K)$ with the
measurable system of boundary measures $\{\nu_x\},\,x\in\T_\ni$ arising from
\thmref{th:exist}. The fact that the measures $\nu_x$ come from the counting
measures on spheres rescaled by powers of the constant $m$ immediately implies

\begin{thm}
Under conditions of \thmref{th:exist} the system of boundary measures
$\{\nu_x\}$ on the treed equivalence relation $(\T_\ni,\P',\R,\K)$ is
conformal with the exponent $\l=\log m$.
\end{thm}

In view of \thmref{th:hp} we now obtain

\begin{thm}
Let $p=(p_k)$ and $q=(q_k)$ be two distributions satisfying conditions
\eqref{eq:stand} and condition \textup{(i)} from \thmref{th:KS}, and such that
they have the same mean
$$
m=\sum k\, p_k =\sum k\, q_k \;.
$$
Denote by $\P'$ and $\Q'$ the respective augmented Galton--Watson measures on
$\T_\ni$, and let measures $\wt\P'$ and $\wt\Q'$ on the boundary bundle
$\wt\T_\ni$ be the boundary extensions of $\P'$ and $\Q'$ determined by the
respective systems of boundary measures from \thmref{th:exist}. Then the image
of the product measure $\wt\P'\times\wt\Q'$ under the map $(T,T')\mapsto T\du
T'$ is an $\R$-invariant finite measure on $\GG_\ni$.
\end{thm}

By \thmref{th:am} it implies

\begin{cor}
If the distributions $\{p_k\}$ and $\{q_k\}$ are in addition finitely
supported, then the horospheric product $T\du T'$ of
$\wt\P'\times\wt\Q'$-a.e.\ pair of pointed at infinity rooted trees $(T,T')$
is amenable.
\end{cor}

\begin{rem}
It is well-known that the horospheric product of two homogeneous trees is
amenable if and only if they have the same degrees (e.g., see \cite{Woess05}).
This Corollary can be considered as an analogue of this result. Actually, for
proving it one does not need all the machinery above. Indeed, if the
distributions $\{p_k\}$ and $\{q_k\}$ have a common point $t$ in their
supports, then the corresponding Galton--Watson trees with probability~1
contain as subgraphs arbitrarily large balls of the homogeneous tree
$\Tb_{t+1}$. An easy estimate then shows that products of these subgraphs will
produce F{\o}lner sets in the horospheric product (cf. \cite{Sobieczky09}). In
fact, a similar argument works also in the situation when the convex hulls of
these supports intersect, i.e.,
$$
[\min\{p_k\},\max\{p_k\}] \cap [\min\{q_k\},\max\{q_k\}] \neq \varnothing \;.
$$
[It would be interesting to study the applicability of this argument to other
invariant measures on $\T_\ni$.] On the other hand, if these intervals do not
intersect, then, \emph{mutatis mutandis}, $\min\{p_k\}\ge\max\{q_k\}+1$, which
means that in the horospheric product all vertex degrees of one multiplier
will be a.s.\ strictly less than all vertex degrees of the other multiplier.
By comparing the simple random walk on the horospheric product of two such
trees with an appropriate biased simple random walk on $\Z$, one can conclude
that in this case the return probabilities decay exponentially, and therefore
the horospheric product is non-amenable.
\end{rem}

\begin{rem}
A connected graph of bounded geometry $\Ga$ is called \emph{strongly amenable}
if it admits a F{\o}lner sequence consisting of connected sets which all contain
a chosen reference vertex. Otherwise $\Ga$ is said to be \emph{weakly
non-amenable} or to have the \emph{anchored expansion property}
\cite{Thomassen92,Benjamini-Lyons-Schramm99,Haggstrom-Schonmann-Steif00}. It
was proved in \cite{Sobieczky09} that horospheric products of percolation
subtrees in a homogeneous tree are a.s.\ strongly amenable. It would be
interesting to address this problem for more general random horospheric
products,
\end{rem}

\bibliographystyle{amsalpha}
\bibliography{D:/Sorted/MyTex/mine}

\enddocument

\bye